\documentclass[12pt,leqno]{article}
\usepackage{amssymb,amsmath,amsthm}
\usepackage[alphabetic,lite]{amsrefs}
\usepackage[all,2cell,ps]{xy}
\numberwithin{equation}{section}
\usepackage{ubshv}

\usepackage{graphicx}

\begin{document}
\title{A lemma for microlocal  sheaf theory in the $\infty$-categorical setting}
%\date{March-April 2014}

\author{Marco Robalo and Pierre Schapira}
\maketitle

\begin{abstract}
Microlocal  sheaf theory of~\cite{KS90} makes an essential use of an extension lemma for sheaves due to Kashiwara, and this lemma is based on a criterion of the same author giving conditions in order that a functor defined in $\R^\rop$ with values in the category $\Sets$ of sets be constant. 

In a first part of this paper,  using classical tools, we show how to generalize the extension lemma to the case of the unbounded derived category. 

In a second part, we extend Kashiwara's  result on constant functors by replacing the category $\Sets$ with  the 
$\infty$-category of spaces and apply it to generalize the extension lemma to $\infty$-sheaves, the $\infty$-categorical version of sheaves. 

Finally, we define the micro-support of sheaves with values in a stable $(\infty,1)$-category.
\end{abstract}

%\tableofcontents

\section{Introduction}
Microlocal sheaf theory appeared in~\cite{KS82} and was developed in \cite{KS85,KS90}. However, this theory is constructed in the framework of the bounded (or bounded from below) derived category of sheaves $\Derb(\cor_M)$ on a real manifold $M$, for a commutative unital ring $\cor$, and it appears necessary in various problems to extend the theory to the unbounded derived category of sheaves $\RD(\cor_M)$. See in particular~\cite{Ta08, Ta15}. 

A crucial result in this theory is~\cite{KS90}*{Prop.~2.7.2}, that we call here the ``extension lemma''. This lemma, which first appeared in~\cite{Ka75,Ka83}),  asserts that if one has an increasing family 
of open subsets $\{U_s\}_{s\in\R}$ of a topological Hausdorff space $M$  and an object $F$ of  $\Derb(\cor_M)$ such that the cohomology of $F$  on  $U_s$ extends through the boundary of $U_s$ for all $s$, then $\rsect(U_s;F)$ is constant with respect to $s$. A basic tool for proving this result is the ``constant functor criterion'', again due to Kashiwara, a result which gives a condition in order that a functor $X\cl\R^\rop\to\Sets$  is constant, where $\Sets$ is the category of sets in a given universe.

In~\S~\ref{sect:unbounded} we generalize the extension lemma to the unbounded setting, that is, to objects of $\RD(\cor_M)$. 
Our proof is  rather elementary and is based on the tools of~\cite{KS90}. This generalization being achieved, the reader can persuade 
himself that most of the 
results, such as the functorial behavior of the micro-support, of~\cite{KS90} extend to the unbounded case.

Next, we consider an higher categorical generalization of this result. In~\S~\ref{sect:constfct} we generalize the constant functor criterion to the case where the $1$-category $\Sets$ is replaced with the $\infty$-category $\Spaces$ of spaces. Using this new tool, in~\S~\ref{section:deformationlemmastable}, we generalize the extension lemma for $\infty$-sheaves with values in any stable compactly generated $\infty$-category $\shd$. When $\shd$ is the $\infty$-category $\mdinf[\cor_M]$ of $\infty$-sheaves of unbounded complexes of $\cor$-modules we recover the results of~\S~\ref{sect:unbounded}.

Finally, in~\S~\ref{section:microsupportstable} we define the micro-support of any $\infty$-sheaf $F$ with general stable higher coefficient.

\begin{remark} 
After this paper has been written, David Treumann informed us of the result of  Dmitri Pavlov~\cite{Pa16}
who generalizes Kashiwara's ``constant functor criterion'' to the case where the functor takes values in the $\infty$-category of spectra. Note that Corollary \ref{cor:126} below implies Pavlov's result on spectra.
\end{remark}

\section{Unbounded derived category of sheaves}\label{sect:unbounded}
Let $\Sets$ denote the category of sets, in a given universe $\shu$. In the sequel, we consider $\mathbb{R}$ as a category with the morphisms being given by the natural order $\leq$.

We first recall a result due to M.~Kashiwara (see~\cite{KS90}*{\S~1.12}).

\begin{lemma}[{The constant functor criterion}]\label{le:1260}
Consider a functor $X\cl \R^\rop\to\Sets$. 
Assume that for each $s\in\R$  
\eq 
&&\indlim[t>s]X_t\isoto X_s\isoto\prolim[r<s]X_r.\label{eq:1260}
\eneq
Then the functor $X$ is constant. 
\end{lemma}

Let $\cor$ denotes a unital ring and denote by $\md[\cor]$ the abelian Grothendieck category of $\cor$-modules.  Set for short 
\eqn
\RC(\cor)&\eqdot&\RC(\md[\cor]),\mbox{ the category of chains complexes of $\md[\cor]$},\\
\quad\RD(\cor)&\eqdot&\RD(\md[\cor])\mbox{ the (unbounded) derived category of $\md[\cor]$}.
\eneqn
We look at the ordered set $(\R,\leq)$ as a category and 
consider a functor  $X\cl\R^\rop\to \RC(\cor)$. We write for short $X_s=X(s)$.
% and denote by $\rho_{s,r}\cl X_s\to X_r$ the morphism  associated with $r\leq s$. 

The next result is a variant on Lemma~\ref{le:1260} and the results of~\cite{KS90}*{\S~1.12}.
\begin{lemma}\label{pro:constantfct}\label{le:126a}
Assume that%for each $r,s\in\R$  with $r\leq s$, each $k,j\in\Z$, 
\eq
&&\mbox{ for any $k\in\Z$, any $r\leq s$ in $\R$, the map $X^k_s\to X^k_r$ is surjective,}\label{eq:126a}\\
&&\mbox{ for any $k\in\Z$, any $s\in\R$, $X^k_s\isoto\prolim[r<s]X^k_r$,}\label{eq:126b}\\
&&\mbox{ for any $j\in\Z$, any $s\in\R$,  $\indlim[t>s] H^j(X_t)\isoto H^j(X_s)$.}\label{eq:126c}
\eneq
Then for any $j\in\Z$, $r,s\in\R$  with $r\leq s$,  one has the isomorphism $H^j(X_t)\isoto H^j(X_s)$.
In other words, for all $j\in\Z$, the functor $H^j(X)$ is constant. 
\end{lemma}

\begin{proof}
%In view of~\eqref{eq:126c} and~\cite{KS90}*{Prop.~1.12.6}, it is enough to prove for all $j\in\Z$, all $s\in\R$:
Consider the assertions for all $j\in\Z$, all $r,s\in\R$ with $r\leq s$:
\eq
&&\mbox{for any $j\in\Z$, $s\in\R$, the map  $H^j(X_s)\to \prolim[r<s] H^j(X_r)$ is surjective},\label{eq:pib1}\\
&&\mbox{for any $j\in\Z$, $r\leq s$, the map  $H^j(X_s)\to  H^j(X_r)$ is surjective},\label{eq:pib3}\\
&&\mbox{for any $j\in\Z$, $s\in\R$, the map $H^j(X_s)\to \prolim[r<s] H^j(X_r)$ is bijective}.\label{eq:pib2}
\eneq
Assertion~\eqref{eq:pib1} follows from hypotheses~\eqref{eq:126a} and~\eqref{eq:126b} by applying~\cite{KS90}*{Prop.~1.12.4~(a)}. 

\spa
Assertion~\eqref{eq:pib3} follows from~\eqref{eq:pib1} and hypothesis~\eqref{eq:126c} in view of~\cite{KS90}*{Prop.~1.12.6}.

\spa
It follows from~\eqref{eq:pib3} that for any $j\in\Z$ and  $s\in\R$, the projective system $\{H^j(X_r)\}_{r<s}$ satisfies the Mittag-Leffler condition. We get~\eqref{eq:pib2} by using~\cite{KS90}*{Prop.~1.12.4~(b)}, . 

\spa
To conclude, apply~\cite{KS90}*{Prop.~1.12.6}, using~\eqref{eq:126c} and~\eqref{eq:pib2}.
\end{proof}

\begin{theorem}[{The non-characteristic deformation lemma}]\label{th:272a}
Let\footnote{St{\'e}phane Guillermou informed us that Claude Viterbo has obtained some time ago a similar result (unpublished). }
 $M$ be a Hausdorff space 
and let $F\in \RD(\cor_M)$. Let $\{U_s\}_{s\in\R}$ be a family of open subsets of $M$. We assume
\banum
\item
for all $t\in\R$, $U_t=\bigcup_{s<t}U_s$,
\item
for all pairs $(s,t)$ with $s\leq t$, the set $\ol{U_t\setminus U_s}\cap\supp F$ is compact,
\item
setting $Z_s=\bigcap_{t>s}\ol{(U_t\setminus U_s)}$, we have for all pairs $(s,t)$ with $s\leq t$ and all $x\in Z_s$, 
$(\rsect_{X\setminus U_t}F)_x\simeq0$. 
\eanum
Then we have the isomorphism in $\RD(\cor)$, for all $t\in\R$
\eqn
&&\rsect(\bigcup_{s}U_s;F)\isoto\rsect(U_t;F).
\eneqn
\end{theorem}
We shall adapt the proof of  \cite{KS90}*{Prop.~2.7.2}, using Lemma~\ref{le:126a}.
\begin{proof}
(i) Following loc.\ cit., we shall first prove the isomorphism:
\eq\label{eq:isoas}
&&(a)^s\cl \indlim[t>s]H^j(U_t;F)\isoto H^j({U_s};F), \text{ for all } j
\eneq
Replacing $M$ with $\supp F$, we may assume from the beginning that $\ol{U_t\setminus U_s}$ is compact. For $s\leq t$, consider the distinguished triangle
\eqn
&&(\rsect_{M\setminus U_t}F)\vert_{Z_s}\to(\rsect_{M\setminus U_s}F)\vert_{Z_s}\to(\rsect_{U_t\setminus U_s}F)\vert_{Z_s}\to[+1].
\eneqn
The two first terms are $0$ by hypothesis (c).  Therefore $(\rsect_{U_t\setminus U_s}F)\vert_{Z_s}\simeq0$ and we get
\eqn
0&\simeq&H^j(Z_s;\rsect_{U_t\setminus U_s}F)\simeq\indlim[U\supset Z_s]H^j(U\cap U_t;\rsect_{M\setminus U_s}F), \text{ for all } j
\eneqn
where $U$ ranges over the family of open neighborhoods of $Z_s$. 

For any such $U$ there exists $t'$ with $s<t'\leq t$ such that $U\cap U_t\supset U_{t'}\setminus U_s$. Therefore,
\eqn
&&\indlim[t, t>s] H^j(U_t;\rsect_{M\setminus U_s}F)\simeq0 \text{ for all } j.
\eneqn
By using the distinguished triangle $\rsect_{M\setminus U_s}F\to F\to\rsect_{U_s}F\to[+1]$, 
we get~\eqref{eq:isoas}.  

\spa
(ii) We shall follow~\cite{KS06}*{Prop. 14.1.6, Th. 14.1.7} and recall that if $\shc$ is a Grothendieck category, then any object of $\RC(\shc)$ is qis to a homotopically injective object whose components are injective. Hence, given $F\in \RD(\cor_M)$, we may represent it by a  homotopically injective object $F^\scbul\in\RC(\cor_M)$ whose components $F^k$ are injective. 
Then $\rsect(U_s;F)$ is represented by $\sect(U_s;F^\scbul)\in\RC(\cor)$. 
Set
\eqn
&& X^k_s=\sect(U_s;F^k), \quad X_s=\sect(U_s;F^\scbul).
\eneqn
Then ~\eqref{eq:126a} is satisfied since $F^k$ is flabby, 
\eqref{eq:126b} is satisfied since $F^k$ is a sheaf and 
\eqref{eq:126c} is nothing but~\eqref{eq:isoas}. 

\spa
(iii) To conclude,   apply Lemma~\ref{le:126a}.
\end{proof}

\section{The constant functor criterion for $\Spaces$}\label{sect:constfct}

\subsection{On $\infty$-categories}
The aim of this subsection is essentially notational and references  are made to~\cite{Lu09, Lu16}. We use Joyal's quasi-categories to model $(\infty,1)$-categories. If not necessary we will simply use the terminology $\infty$-categories.

Denote by $\infCat$ the $(\infty,1)$-category of all $(\infty,1)$-categories in a given universe $\shu$ and by 
$\oneCat$ the $1$-category of all $1$-categories in $\shu$. 

To $\shc\in\oneCat$, one associates its nerve, $N(\shc)\in\infCat$. 
Denoting by $N(\oneCat)$ the image of $\oneCat$ by $N$,
the embedding $\iota\cl N(\oneCat)\into\infCat$ admits a left adjoint $\mathrm{h}$. We get the functors :
\eqn
\xymatrix{
&& \mathrm{h}\cl\quad \infCat\ar@<.5ex>[r]&\ar@<.5ex>[l] N(\oneCat) \quad \cl\iota
}\eneqn
Hence, $\mathrm{h}\circ\iota\simeq\id_1$ and there exists a natural morphism of $\infty$-functors 
$\id_\infty\to\iota\circ\mathrm{h}$, where $\id_1$ and $\id_\infty$ denote the identity functors of the categories $\oneCat$ and $\infCat$, respectively. %In the sequel, we do not write $\iota$. 

Looking at  $\infCat$ as a simplicial set,  its degree $0$ elements are the $(\infty,1)$-categories, its degree $1$ elements are the $\infty$-functors, etc. Hence
the functor $\mathrm{h}$ sends a $(\infty,1)$-category to a usual category, an $\infty$-functor to a usual functor, etc. 
Its sends a stable  $(\infty,1)$-category to a triangulated category where the distinguished triangles are induced by the cofiber-fiber sequences. Moreover, it sends an $\infty$-functor to a triangulated functor, etc. See \cite[1.1.2.15]{Lu16}.

Let $\Spaces$ (resp. $\Spaces_*$) denote the $(\infty,1)$-category of spaces (resp. pointed spaces) \cite[1.2.16.1]{Lu09}. Informally, one can think of $\Spaces$ as a simplicial set whose vertices are CW-complexes, 1-cells are continuous maps, 2-cells are homotopies between continuous maps, etc. Recall 
that $\Spaces$ admits small limits and colimits in the sense of \cite[1.2.13]{Lu09}.
Moreover, by Whitehead's theorem, a map $f:X\to Y$ in $\Spaces$ is an equivalence if and only if the induced map
$\pi_0(f):\pi_0(X)\to \pi_0(Y)$ is an isomorphism of sets and for every base point $x\in X$, the induced maps
$\pi_n(X,x)\to \pi_n(Y, f(x))$ are isomorphisms for all $n\geq 1$. 

It is also convenient to recall the existence of a Grothendieck construction for $(\infty,1)$-categories. Namely, for any $(\infty,1)$-category $\shc$ we have an equivalence of $(\infty,1)$-categories 
\eq \label{eq:grothendieckconstruction}
\mathrm{St}:\,\infCat^\mathrm{cart}/\shc \simeq \mathrm{Fun}(\shc^\rop, \infCat)\, :\mathrm{Un}
\eneq
where on the l.h.s we have the $(\infty,1)$-category of $\infty$-functors $\shd\to \shc$ that are cartesian fibrations and functors that preserve cartesian morphisms (see \cite[Def. 2.4.1.1]{Lu09}), and on the r.h.s. we have the $(\infty,1)$-category of $\infty$-functors from $\shc^\rop$ to $\infCat$. See \cite[3.2.0.1]{Lu09}. The same holds for diagrams in $\Spaces$, where we find 
\eq \label{eq:grothendieckconstruction2}
\mathrm{St}:\infCat^\mathrm{Right-fib}/\shc \simeq \mathrm{Fun}(\shc^\rop, \Spaces)\, :\mathrm{Un}
\eneq
where this time on the l.h.s. we have the $(\infty,1)$-category of $\infty$-functors $\shd\to \shc$ that are right fibrations. See \cite[2.2.1.2]{Lu09}. 
The equivalence \eqref{eq:grothendieckconstruction2} will be useful for the following reason: for any diagram $X:\shc^\rop\to\Spaces$, its limit in $\Spaces$ can be identified with the space of sections of $\mathrm{Un}(X)$ \cite[3.3.3.4]{Lu09}:
\eq \label{eq:grothendieckconstructionlimt}
\lim X\simeq \mathrm{Map}_{\shc}(\shc, \mathrm{Un}(X)).
\eneq

\subsection{A criterion for a functor to be constant}
In this subsection, we generalize \cite{KS90}*{Prop.~1.12.6} to the case of an $\infty$-functor.  
Let $X\cl \R^\rop\to \Spaces$ be an $\infty$-functor. We set 
\eq \label{eq:maps}
&& X_s=X(s),\quad \rho_{s,t}\cl X_t\to X_s \, (s\leq t).
\eneq

\begin{lemma}\label{le:126}
Let $X\cl \R^\rop\to \Spaces$ be an $\infty$-functor. Assume that for each $s\in\R$, the natural morphisms in $\Spaces$
\eq \label{le:126-eq}
&&\colim[s<t]X_t\to X_s\to\plim[r<s]X_r
\eneq
\noindent both are equivalences. Then  for every $t\geq s$, the morphism $X_t\to X_s$ in $\Spaces$ is an equivalence.
\end{lemma}
The proof adapts to the case of $\Spaces$ that of \cite{KS90}*{Prop.~1.12.6} and will also use this result.

\begin{proof}
\hfill \\
\noindent (Step I) It is enough to prove that for each $c\in \mathbb{R}$, the restriction of $X$ to $\mathbb{R}_{<c}$ is constant. 

\vspace{0.5cm}

\noindent (Step II: Choosing base points) Let $c\in \mathbb{R}$ and let again $X$ denote the restriction of $X$ to $\mathbb{R}_{<c}$. The hypothesis 
$$
\plim[s<c] X_s\simeq X_c
$$
ensures that the choice of a base point in $X_c$ determines a compatible system of base points up to homotopy at every $X_s$ with $s<c$, i.e. the choice of a 2-simplex $\sigma:\Delta^2\to \infCat$
\eq\label{eq:lifting}
\xymatrix{
& \Spaces_{*}\ar[d]\\
\mathrm{N}(\R^\rop_{<c}) \ar@{-->}[ru]^{\overline{X}}\ar[r]_X& \Spaces. \ar@{}[ul]_(.30){\sigma}
}
\eneq
For the reader's convenience we explain how to construct the 2-simplex $\sigma$. Thanks to \eqref{eq:grothendieckconstructionlimt}, the limit $\plim[s<c] X_s$ can be identified with the category of sections of the right fibration $\mathrm{Un(X)}\to \mathrm{N}(\R_{<c})$. Therefore, the choice of a base point in $X_c$ provides a section of $\mathrm{Un}(X)$, which we can see as a map from the trivial cartesian fibration $\mathrm{Id}:\mathrm{N}(\R_{<c})\to \mathrm{N}(\R_{<c})$ to $\mathrm{Un}(X)$. Its image via the functor $\mathrm{St}$ of \eqref{eq:grothendieckconstruction2}  provides the lifting \eqref{eq:lifting}.

\vspace{0.3cm}

Recall that the forgetful functor $\Spaces_*\to \Spaces$ preserves filtrant colimits and all small limits and is conservative. Therefore the hypothesis are also valid for $\overline{X}$.

\vspace{0.3cm}

\noindent (Step III: working with a fixed choice of base points)
Choose any lifting $\overline{X}$ of $X$.  We have for each $n\in\N$, $s\in\R_{<c}$, a short exact sequence\footnote{of groups when $n\geq 1$ and pointed sets when $n=0$.}, called the Milnor exact sequence (see for instance~\cite[Prop. 2.2.9]{MP12}\footnote{The Milnor exact sequence is usually define for $\mathbb{N}^\rop$-towers. However, the argument works for $\mathbb{R}^\rop$-towers as the inclusion $\mathbb{N}^\rop\subseteq \mathbb{R}^\rop$ is cofinal.}):
 \eq\label{eq:milnor}
 &&0\to R^1\plim[r< s] \pi_{n+1}(\overline{X}_r)\to \pi_n( \plim[r< s] \overline{X}_r)\to \plim[r< s] \pi_n(\overline{X}_r)\to 0.
 \eneq
Under the hypothesis of the lemma, we get short exact sequences:
  \eq\label{eq:milnor2}
 &&0\to R^1\plim[r< s] \pi_{n+1}(\overline{X}_r)\to \pi_n(\overline{X}_s) \to \plim[r< s] \pi_n(\overline{X}_r)\to 0.
 \eneq

\spa
 For each $n\geq 0$, each $s, t\in\R_{\leq x}$ with $t\geq s$, we shall prove:
 \eq
&&\mbox{the map }\colim[c>t>s]\pi_n(\overline{X}_t)\to \pi_n(\overline{X}_s) \mbox{ is bijective,}\label{eq:pia2}\\
&&\mbox{the map } \pi_n(\overline{X}_s)\to \plim[r<s] \pi_n(\overline{X}_r)\mbox{ is surjective,}\label{eq:pib22}\\
&&\mbox{the map } \pi_n(\overline{X}_t)\to \pi_n(\overline{X}_s)\mbox{ is surjective,}\label{eq:pic2}\\
&&\mbox{the map } \pi_n(\overline{X}_s)\to \plim[r<s] \pi_n(\overline{X}_r)\mbox{ is bijective.}\label{eq:pid2}
  \eneq

\spa
Assertion~\eqref{eq:pia2} follows from the hypothesis, the fact that the system $\{t: c>t>s\}$ is cofinal in $\{t: t>s\}$ and the fact that $\pi_n$ commutes with filtrant colimits for $n\geq 0$. \\

Assertion~\eqref{eq:pib22} follows from~\eqref{eq:milnor2}. \\

Let us prove~\eqref{eq:pic2}. 
By the surjectivity result in~\cite{KS90}*{Prop.~1.12.6}, it is enough to prove  the surjectivity of  
$\colim[c>t>s]\pi_n(\overline{X}_t)\to \pi_n(\overline{X}_s)$ and  $\pi_n(\overline{X}_s)\to \plim[r<s] \pi_n(\overline{X}_r)$ for all $s\in\R_{<c}$, which follows from~\eqref{eq:pia2}
and~\eqref{eq:pib22}.\\

By \eqref{eq:pic2} we know that the projective systems  $\{\pi_n(\overline{X}_r)\}_{r<s}$ satisfy the Mittag-Leffler condition for all $n\geq 0$, $s< c$. Therefore,  $R^1\plim[r< s] \pi_{n+1}(\overline{X}_r)\simeq 0$ for all $n$, all $s\in\R_{<c}$ and~\eqref{eq:pid2} follows from~\eqref{eq:milnor2}. Therefore, we have isomorphisms for every $n\geq 0$

\eqn
&&\colim[s<t<c] \pi_n(\overline{X}_t)\simeq \pi_n(\overline{X}_s)\simeq \plim[r<s]\pi_n(\overline{X}_r).
\eneqn
Applying~\cite{KS90}*{Prop.~1.12.6}, we get that 
the diagram of sets $s\mapsto \pi_n(\overline{X}_s)$ is constant for every $n$. 

\vspace{0.3cm}

\noindent (Step IV: End of the Proof)
The conclusion of Step III holds for any lifting $\overline{X}$ of the restriction of $X$ to $\mathbb{R}_{<c}$. As the result holds for $n=0$, the diagram $s< c \mapsto \pi_0(X_s)$ is also constant, seen as a diagram of sets rather than pointed sets.

\vspace{0.3cm}

To conclude one must show that for any $n\in\N$, $t\geq s\in \mathbb{R}_{<c}$ and for every choice of a base point $y$ in $X_t$, the induce maps
\eq \label{eq:mapshigherhomotopygroups}
&& \rho^n_{s,t}\cl \pi_n(X_t, y)\to \pi_n(X_s, \rho^n_{s,t}(y)).
\eneq
are bijective. Since, for $\alpha<c$, $\alpha\mapsto \pi_0(X_\alpha)\in \Sets$ is constant, choosing $l \in \R$ with $t<l < c$, $y$ determines a unique element $\bar{y}$ in $\pi_0(X_l)$ and again using the hypothesis $X_l \simeq \lim_{r< l}X_r$, the choice of a representative for $\bar{y}$ determines an homotopy compatible system of base points at every $X_r$ for $r<l$ and therefore a new lifting $\overline{X}$ of the restriction of $X$ to $\R_{<l}$ whose associated base point at $X_t$ is a representative of $y$ and the composition with $\pi_n$ provides the maps \eqref{eq:mapshigherhomotopygroups}. By \eqref{eq:pia2}, \eqref{eq:pib22}, \eqref{eq:pic2}, \eqref{eq:pid2} and ~\cite{KS90}*{Prop.~1.12.6} the maps \eqref{eq:mapshigherhomotopygroups} are isomorphisms.
This conclusion holds for any $c\in \mathbb{R}$ and thus for any $t\geq s$ in $\R^\rop$.
\end{proof}

We refer to \cite[5.5.7.1]{Lu09}  for the notion of presentable compactly generated $(\infty,1)$-category.

\begin{corollary}\label{cor:126}
Let $\shc$ be a presentable compactly generated $(\infty,1)$-category and let  $X\cl \R^\rop\to \shc$ be an $\infty$-functor. Assume that for each $s\in\R$, the natural morphisms
\eqn \label{le:126-eqb}
&&\colim[s<t]X_t\to X_s\to\plim[r<s]X_r
\eneqn
both are equivalences. Then for any $t\geq s$ the induced map $X_t\to X_s$ is an equivalence.
\end{corollary}
\begin{proof}
Apply the Lemma~\ref{le:126} to all mapping spaces $\mathrm{Map}(Z, X_t)$ for each compact object $Z$.
\end{proof}

\begin{remark}
This result does not apply to $\shc=\R^\rop$ and $X$ the identity functor. Indeed, $\R^\rop$ is not compactly generated in the sense of \cite[5.5.7.1]{Lu09}.
\end{remark}

\begin{remark}
As noticed by M.Porta, the category $\mathbb{R}^{\mathrm{op}}$ being contractible, the condition that for any $t\geq s$ the induced map $X_t\to X_s$ is an equivalence, is equivalent to $X$ being a constant functor.
\end{remark}

\section{Micro-support }
\subsection{The non-characteristic deformation lemma with stable coefficients}
\label{section:deformationlemmastable}
In this subsection, we generalize \cite{KS90}*{Prop.~2.7.2} and Theorem \ref{th:272a} to more general coefficients.  Let $\shd$ be a presentable compactly generated stable $(\infty,1)$-category. Given a topological space $M$ we denote by $\Op_M$ its category of open subsets. One defines an higher categorical version of sheaves on $M$ as follows. Let $\mathrm{Psh}(M, \shd)$ denote the $(\infty,1)$-category of $\infty$-functors 
$$\mathrm{N}(\Op_M^{\mathrm{op}})\to \shd$$

See \cite[1.2.7.2, 1.2.7.3]{Lu09}. The category $\Op_M$ is equipped with a Grothendieck topology whose covering of $U$ are the families $\{U_i\}_i$ such that $U_i\subseteq U$  and $\bigcup_\alpha U_\alpha = U$. We let $\mathrm{Sh}(M, \shd)^\wedge$  denote the full subcategory of $\mathrm{Psh}(M, \shd)$ spanned by those functors that satisfy the sheaf condition and are hypercomplete. See \cite[6.2.2]{Lu09} and \cite[Section 1.1]{DAGV} for the theory of $\infty$-sheaves and \cite[6.5.2, 6.5.3, 6.5.4]{Lu09} for the notion of hypercomplete. The $(\infty,1)$-category $\mathrm{Sh}(M, \shd)^\wedge$ is again a stable compactly generated $(\infty,1)$-category and when $M=\rmpt$, one recovers $\mathrm{Sh}(M, \shd)^\wedge\simeq \shd$.

The usual pullback and push-forward functorialites can be lifted to the higher categorical setting and are given by exact functors. See for instance the discussion in \cite[Section 2.4]{YP16}. Let $j_U\cl U\into M$ be an open embedding and let $a_M\cl M\to\rmpt$ be the map from $M$ to one point. We introduce the notations
\eqn
&&\isect(U;\scbul)\eqdot \oim{a^\infty_M}\circ \oim{j_U^\infty}\circ \opb{{j_U^\infty}}\cl \mathrm{Sh}(M, \shd)^\wedge\to\shd,
\eneqn
where $\oim{a^\infty_M}$, $\oim{j_U^\infty}$, $\opb{{j_U^\infty}}$ are the  direct and inverse image functors for $(\infty,1)$-categories of sheaves.  
 If $Z$ is  a closed subset of $U$, using  the cofiber-fiber sequence associated to $\isect(U;\scbul)\to\isect(U\setminus Z;\scbul)$,  we define 
\eqn
&&\isect_Z(U;\scbul)\cl \mathrm{Sh}(M, \shd)^\wedge\to\shd.
\eneqn

The following result generalizes \cite{KS90}*{Prop.~2.7.2} and Theorem \ref{th:272a} to any context of sheaves with stable coefficients:

\begin{theorem}[{The non-characteristic deformation lemma for stable coefficients}]\label{th:272}
Let $M$ be a Hausdorff space 
and let $F\in \mathrm{Sh}(M, \shd)^\wedge$. Let $\{U_s\}_{s\in\R}$ be a family of open subsets of $M$. We assume
\banum
\item
for all $t\in\R$, $U_t=\bigcup_{s<t}U_s$,
\item
for all pairs $(s,t)$ with $s\leq t$, the set $\ol{U_t\setminus U_s}\cap\supp F$ is compact,
\item
setting $Z_s=\bigcap_{t>s}\ol{(U_t\setminus U_s)}$, we have for all pairs $(s,t)$ with $s\leq t$ and all $x\in Z_s$, 
$(\isect_{X\setminus U_t}F)_x\simeq0$. 
\eanum
Then we have the equivalences in $\shd$, for all $s,t\in\R$
\eqn
&&\isect(\bigcup_{s}U_s;F)\isoto\isect(U_t;F).
\eneqn
\end{theorem}
We shall almost mimic the proof of  \cite{KS90}*{Prop.~2.7.2}.
\begin{proof}
(i) We shall  prove the equivalences
\eqn
&&(a)^t\cl \plim[s<t]\isect(U_s;F)\isofrom\isect(U_t;F),\\
&&(b)^s\cl \colim[t>s]\isect(U_t;F)\isoto\isect({U_s};F).
\eneqn

\spa
(ii) Equivalence $(a)^t$ is always true by hypothesis (a). Indeed one has 
$\cor_{U_s}\simeq\prolim[r<s]\cor_{U_r}$,  which implies $\plim[s<t]\isect_{U_s}F\isofrom\isect_{U_t}F$, and the result follows since the direct image functor commutes with $\plim$ (because it is a right adjoint).

\spa
(iii)  The proof of the  equivalence  $(b)^s$ for all $s$ is formally the same as the proof of~\eqref{eq:isoas} which itself mimics that of~\cite{KS90}*{Prop.~2.7.2} and we shall not repeat it. 

\spa
To conclude,
 apply Corollary~\ref{cor:126} to $\shd$.
 \end{proof}
 
\begin{remark}
Let $\cor$ denote a commutative unital ring. Theorem \ref{th:272} recovers  the result of Theorem \ref{th:272a} in the particular case where $\shd$ is the $\infty$-version of the derived category of $\cor$, which we will denote as $\mdinf[\cor]$. We define it as follows: let $\RC(\cor)$ denote the 1-category of (unbounded) chain complexes over $\cor$. One considers the nerve $\mathrm{N}(\RC(\cor))$ and settles $\mdinf[\cor]$ as the localization $\mathrm{N}(\RC(\cor))[\mathcal{W}^{-1}]$ along the class of edges $\mathcal{W}$ given by quasi-isomorphisms of complexes. This localization is taken inside the theory of $(\infty,1)$-categories. See \cite[4.1.3.1]{Lu16}. The homotopy category $\mathrm{h}(\mdinf[\cor])$ is canonically equivalent to $\RD(\cor)$ by the universal properties of the higher and the classical localizations. In this case we settle the notation 
$$\mdinf[\cor_M]:=\mathrm{Sh}(M, \mdinf[\cor])^\wedge.$$
The homotopy category of $\mdinf[\cor_M]$ recovers the usual derived category of (unbounded) complexes of sheaves of $\cor$-modules, $\RD(\cor_M)$. This follows from \cite[Prop.  2.1.8]{Lu11} and the definition of hypercomplete sheaves.  When $M=\rmpt$, one recovers $\mdinf[\cor]$ and $\RD(\cor)$, respectively. 
\end{remark}

\begin{remark} If we assume that $M$ is a topological manifold (therefore homotopy equivalent to a CW-complex), then 
$ \mathrm{Sh}(M, \shd)^\wedge$ is equivalent to $ \mathrm{Sh}(M, \shd)$. In other words, sheaves on topological manifolds are automatically hypercomplete. In particular,  $\mdinf[\cor_M]$ is equivalent  to the higher category $\mathrm{Sh}(M, \mdinf[\cor])$  of $\infty$-sheaves obtained without imposing hyperdescent. To see this we use the fact that every CW-complex can be obtained as a filtered colimit of finite CW-complexes.  Then we combine \cite[7.2.1.12, 7.2.3.6, 7.1.5.8, 6.5.2.13 ]{Lu16}.
\end{remark}

\begin{remark}
In~\cite{KS90}*{Prop.~2.7.2}, $Z_s$ was defined as $Z_s=\ol{\bigcap_{t>s}{(U_t\setminus U_s)}}$, which was a mistake. 
This mistake is already corrected in the  ``Errata'' of: \\
https://webusers.imj-prg.fr/~pierre.schapira/books/.
\end{remark}

\subsection{Micro-support}
\label{section:microsupportstable}
The definition~\cite{KS90}*{Def.~5.1.2} of the micro-support of sheaves immediately extends to  $\infty$-sheaves with stable coefficients.

Let $M$ be a real manifold of class $C^1$ and denote by  $T^*M$ its cotangent bundle.
\begin{definition}\label{def:SS}
 Let $F\in \mathrm{Sh}(M, \shd)$. The micro-support of $F$, denoted $\musupp(F)$, is the closed $\R^+$-conic subset of $T^*M$ defined as follows. 
 For $U$ open in $T^*M$, $U\cap\musupp(F)=\emptyset$ if for any $x_0\in M$ and any
real $C^1$-function $\phi$ on $M$ defined in a neighborhood of $x_0$ 
satisfying $d\phi(x_0)\in U$ and $\phi(x_0)=0$, one has
$(\isect_{\{x;\phi(x)\geq0\}} (F))_{x_0}\simeq0$.
\end{definition}

When $\shd$ is $\mdinf[\cor]$, one recovers the classical definition of the micro-support.
\begin{remark}
As already mentioned in the introduction, Theorem~\ref{th:272a} is the main tool to develop  microlocal sheaf theory in the framework of classical derived categories. We hope that similarly Theorem~\ref{th:272} will be the main tool to develop  microlocal sheaf theory in the  new framework of sheaves with stable coefficients.
\end{remark}

\begin{remark}
In~\cite{KS90}, the micro-support of $F$ was denoted $\SSi(F)$, a shortcut for ``singular support''. 
Some people made the remark that this notation had  very bad historical reminiscences and that is the reason of this new terminology, $\musupp$. 
\end{remark}

%
%\subsection{Micro-support}
%\label{section:microsupportstable}
%The definition~\cite{KS90}*{Def.~5.1.2} of the micro-support of sheaves immediately extends to  $\infty$-sheaves with stable coefficients: 
%
%\begin{definition}\label{def:SS}
%Let $F\in \mathrm{Sh}(M, \shd)^\wedge$ and let $p\in T^*M$. 
%One says that $p\notin\SSi(F)$ if there exists an open neighborhood
%$U$ of $p$ such that for any $x_0\in M$ and any
%real $C^1$-function $\phi$ on $M$ defined in a neighborhood of $x_0$ 
%satisfying $d\phi(x_0)\in U$ and $\phi(x_0)=0$, one has
%$(\isect_{\{x;\phi(x)\geq0\}} (F))_{x_0}\simeq0$.
%\end{definition}
%
%One calls $\SSi(F)$ the micro-support, or singular support, of $F$.
%
%Consider the homotopy functor
%\eqn
%&&h\cl  \mdinf[\cor_M]\to\RD(\cor_M)
%\eneqn
%\begin{proposition}
%Let $F\in\mdinf[\cor_M]$. Then $\SSi(F)=\SSi(hF)$. 
%\end{proposition}
%\begin{proof}

%\end{proof}

\providecommand{\bysame}{\leavevmode\hbox to3em{\hrulefill}\thinspace}

\vspace*{1cm}
\noindent
\parbox[t]{16em}
{\scriptsize{
\noindent
Marco Robalo\\
Sorbonne Universit{\'e}s, UPMC Univ Paris 6\\
Institut de Math{\'e}matiques de Jussieu\\
e-mail: marco.roballo@imj-prg.fr\\
http://webusers.imj-prg.fr/\textasciitilde marco.robalo/
}}
\hspace{0.1cm}
\parbox[t]{16em}
{\scriptsize{
Pierre Schapira\\
Sorbonne Universit{\'e}s, UPMC Univ Paris 6\\
Institut de Math{\'e}matiques de Jussieu\\
e-mail: pierre.schapira@imj-prg.fr\\
http://webusers.imj-prg.fr/\textasciitilde pierre.schapira/}}

\end{document}